\documentclass[11pt,a4paper]{article}
\usepackage[T1]{fontenc}
\usepackage[utf8]{inputenc}
\usepackage[english]{babel}
\usepackage{lmodern,amsmath,amssymb,amsthm}
\usepackage[margin=1in]{geometry}

\newtheorem{theorem}{Theorem}[section]
\newtheorem{lemma}[theorem]{Lemma}

\newtheorem{corollary}[theorem]{Corollary}
\newcommand{\F}{\mathbb F}
\newcommand{\PG}{\mathrm{PG}}
\newcommand{\cQ}{\mathcal Q}
\newcommand{\cW}{\mathcal W}
\newcommand{\cH}{\mathcal H}
\newcommand{\Tr}{\operatorname{Tr}}
\usepackage{microtype}
\usepackage[hidelinks]{hyperref}
\numberwithin{equation}{section}
\newcommand{\doi}[1]{\href{https://doi.org/#1}{doi:\nolinkurl{#1}}}
\title{Line spreads of $\cW(2n-1,q)$, $q$ even}
\author{Daniele Bartoli\textsuperscript{1}\quad 
Antonio Cossidente\textsuperscript{2}\quad Giuseppe Marino\textsuperscript{3}\quad Francesco Pavese\textsuperscript{4}\\[1em]
\parbox{0.94\textwidth}{\centering\small
\textsuperscript{1}Department of Mathematics and Computer Science, University of Perugia, Perugia, Italy.\\
\texttt{daniele.bartoli@unipg.it}\\
\textsuperscript{2}Department of Engineering, University of Basilicata, Potenza, Italy.\\
\texttt{antonio.cossidente@unibas.it}\\
\textsuperscript{3}Department of Mathematics and Applications ``Renato Caccioppoli'', University of Naples Federico II, Naples, Italy.\\
\texttt{giuseppe.marino@unina.it}\\
\textsuperscript{4}Department of Mechanics, Mathematics and Management, Politecnico di Bari, Bari, Italy.\\
\texttt{francesco.pavese@poliba.it}
}}
\date{}
\hypersetup{pdftitle={Line spread of mathcal W(2n-1, q), q even}}
\begin{document}

\maketitle
\begin{abstract}
Let $q$ be an even prime power and $n\ge2$. We construct line spreads
of the symplectic polar space $\cW(2n-1,q)$ by replacing pairwise
disjoint reguli in a Desarguesian line spread. The construction is
based on partitions of the points of $\PG(n-1,q^2)$ into points of a
nondegenerate Hermitian variety and Baer sublines meeting that variety
in exactly one point. In the plane case,
the partition is invariant under a cyclic group of order $q^2-q+1$.
\end{abstract}

\section{Introduction}\label{sec:introduction}
A line spread of a finite projective space is a set of pairwise
disjoint lines partitioning its point set. A line spread of a
symplectic polar space is required, in addition, to consist of totally
isotropic lines. Spreads and their interaction with classical polar
spaces are treated in \cite{Hirschfeld,HT}; for the rank-two case,
see also \cite{PT}.

The general theory of ovoids and spreads in finite classical polar
spaces provides a broader context for this problem \cite{Thas1981}.
Field reduction also gives a direct construction in certain dimensions:
when $n$ is even, a symplectic space of vector dimension $n$ over
$\F_{q^2}$ yields a line spread of $\cW(2n-1,q)$ over $\F_q$.
The case $n=3$ does not follow from this construction.
Hering's line spreads of $\cW(5,3)$ \cite{Hering1985} are relevant
small-field examples. De Winter and Metsch \cite[Remark~8]{DeWinterMetsch2020}
recall these examples and observe that relatively few line spreads of
$\cW(5,q)$ were known in the literature. This provides
motivation for a uniform construction.

Let $q$ be even, let $n\ge 2$, and let $\cH(n-1,q^2)$ be a
non-degenerate Hermitian variety of $\PG(n-1,q^2)$. Put $\varepsilon=(-1)^n$.
In this paper, we show that the points of $\PG(n-1,q^2)$ can be
partitioned into
\[
 p_n=\frac{(q^{n-2}+\varepsilon)(q^n-\varepsilon)}{q^2-1}
\]
points of $\cH(n-1,q^2)$ and
\[
 b_n=\frac{q^{n-2}(q^n-\varepsilon)}{q+1}
\]
Baer sublines, each meeting $\cH(n-1,q^2)$ in exactly one point,
where $\varepsilon=(-1)^n$.
Under field reduction from $\F_{q^2}$ to $\F_q$, the points of
$\PG(n-1,q^2)$ correspond to the lines of a Desarguesian spread of
$\PG(2n-1,q)$ \cite{LV}, while the points of a Baer subline correspond to a
regulus. As shown in Section~\ref{sec:reduction}, replacing the corresponding $b_n$ reguli by their opposite reguli gives a line spread of $\cW(2n-1,q)$. This spread contains $p_n$ lines on the associated quadric
$\mathcal Q^\varepsilon(2n-1,q)$ and $(q+1)b_n$ lines tangent to it.

There is also a connection with intriguing sets in dual polar graphs.
In the dual polar graph of $\cW(5,q)$, vertices are totally isotropic
planes, with two vertices adjacent when the planes intersect in a line.
For a line spread, take the planes containing at least one spread line.
Each selected plane has $q^2+q$ selected neighbours, while each
unselected plane has $q^2+q+1$ selected neighbours. Consequently, this
set defines a perfect two-colouring with nontrivial eigenvalue $-1$.
In particular, the rank-three spreads constructed here yield such
colourings for every even $q$; see \cite[Example~7]{DeWinterMetsch2020} for this graph construction.

Section~\ref{sec:reduction} presents the geometric setting.
Section~3 supplies the required partition and completes the construction.
In particular, the plane case uses a cyclic model of the Hermitian curve,
for which background can be found in \cite{BE}.

\section{Quadrics and regulus replacement}\label{sec:reduction}
Let $q$ be an even prime power, let $n\geq 2$, and let $\cW(2n-1,q)$
be a symplectic polar space of $\PG(2n-1,q)$.
Denote by $\cQ^\varepsilon$ a nondegenerate quadric of
$\PG(2n-1,q)$, with $\varepsilon=+1$ if the quadric is hyperbolic
and $\varepsilon=-1$ if it is elliptic, and assume that its associated
polar form defines $\cW(2n-1,q)$.
A line of $\PG(2n-1,q)$ is a line of $\cW(2n-1,q)$ if and only if
it is either tangent to or contained in $\cQ^\varepsilon$, where a
tangent line meets $\cQ^\varepsilon$ in exactly one point.

\begin{lemma}\label{lem:counts}
Let $\mathcal L$ be a line spread of $\cW(2n-1,q)$, and let
$\cQ^\varepsilon$ be a nondegenerate quadric of $\PG(2n-1,q)$
whose associated polar form defines $\cW(2n-1,q)$.
The numbers of lines of $\mathcal L$
contained in and tangent to $\cQ^\varepsilon$ are
\[
\frac{(q^{n-2}+\varepsilon)(q^n-\varepsilon)}{q^2-1}
\qquad\text{and}\qquad
q^{n-2}(q^n-\varepsilon),
\]
respectively.
\end{lemma}

\begin{proof}
Let $x$ and $y$ denote the numbers of lines of $\mathcal L$
contained in and tangent to $\cQ^\varepsilon$, respectively.
Since $\mathcal L$ partitions the points of
$\PG(2n-1,q)$,
\[
x+y=|\mathcal L|=\frac{q^{2n}-1}{q^2-1}.
\]
On the other hand, counting the points of $\cQ^\varepsilon$ on the lines of
$\mathcal L$ gives
\[
(q+1)x+y=|\cQ^\varepsilon|
       =\frac{(q^{n-1}+\varepsilon)(q^n-\varepsilon)}{q-1}.
\]
Subtracting these two identities yields
\[
x=\frac{|\cQ^\varepsilon|-|\mathcal L|}{q}
 =\frac{(q^{n-2}+\varepsilon)(q^n-\varepsilon)}{q^2-1}.
\]
Consequently,
\[
y=|\mathcal L|-x=q^{n-2}(q^n-\varepsilon),
\]
as required.
\end{proof}

Embed $\PG(2n-1,q)$ as a Baer subgeometry $\Sigma$ of
$\PG(2n-1,q^2)$, and let $\sigma$ be the Baer involution of
$\PG(2n-1,q^2)$ fixing
$\Sigma$ pointwise. Let $\Pi$ be an $(n-1)$-dimensional projective
subspace of $\PG(2n-1,q^2)$ such that
\[
\Pi\cap\Sigma=\varnothing
\qquad\text{and}\qquad
\Pi\cap\Pi^\sigma=\varnothing.
\]
Then $\Pi$ and $\Pi^\sigma$ span $\PG(2n-1,q^2)$. For $P\in\Pi$, put $\ell_P=\langle P,P^\sigma\rangle_{q^2}\cap\Sigma$.
The lines $\ell_P$, $P\in\Pi$, form the Desarguesian spread
$\mathcal S$.
For each line $\ell$ of $\cW(2n-1,q)$, let $\overline{\ell}$ denote
its extension to $\PG(2n-1,q^2)$. Assume further that the point set
\[
\bigcup_{\ell\text{ a line of }\cW(2n-1,q)}
\bigl(\overline{\ell}\cap\Pi\bigr)
\]
is a nondegenerate Hermitian variety $\mathcal H(n-1,q^2)$ of $\Pi$.
The lines of $\mathcal S$ corresponding to the points of
$\mathcal H(n-1,q^2)$ partition a nondegenerate quadric
$\cQ^\varepsilon$ of $\Sigma$, with $\varepsilon=(-1)^n$;
thus this quadric is elliptic, $\cQ^-$, if $n$ is odd and
hyperbolic, $\cQ^+$, if $n$ is even.

\begin{theorem}\label{thm:replacement}
Let $\mathcal A\subseteq\mathcal H(n-1,q^2)$ be a set of
\[
\frac{(q^{n-2}+(-1)^n)(q^n-(-1)^n)}{q^2-1}
\]
points, and let $\mathcal B$ be a collection of
\[
\frac{q^{n-2}(q^n-(-1)^n)}{q+1}
\]
Baer sublines, each meeting $\mathcal H(n-1,q^2)$ in exactly one
point, such that $\mathcal A$ and the members of $\mathcal B$
partition the points of $\Pi$.
Then this partition gives rise to a line spread of $\cW(2n-1,q)$.
\end{theorem}

\begin{proof}
Use the Desarguesian spread $\mathcal S$ defined above. For each Baer subline $b\in\mathcal B$, the set
$\mathcal R_b=\{\ell_P:P\in b\}$ is a regulus in a solid of $\Sigma$.
Let $\mathcal R_b'$ be its opposite regulus. The two reguli cover
exactly the same points. Since the members of $\mathcal B$ are pairwise disjoint, the point sets covered by their reguli are pairwise disjoint. Hence
\[
\mathcal L=\{\ell_P:P\in\mathcal A\}\;\cup\!
\left(\bigcup_{b\in\mathcal B}\mathcal R_b'\right)
\]
is a line spread of $\Sigma$.

For $P\in\mathcal A$, the line $\ell_P$ is contained in $\cQ^\varepsilon$
and therefore belongs to $\cW(2n-1,q)$. If
$b\cap\mathcal H(n-1,q^2)=\{P_b\}$, the points covered by
$\mathcal R_b$ meet $\cQ^\varepsilon$ precisely in $\ell_{P_b}$.
Every line of $\mathcal R_b'$ meets $\ell_{P_b}$ in one point and
thus meets $\cQ^\varepsilon$ in exactly one point. Consequently, every line of $\mathcal R_b'$ also belongs to
$\cW(2n-1,q)$, proving the assertion.
\end{proof}

\section{A partition of \texorpdfstring{$\PG(n-1,q^2)$}{PG(n-1, q²)}}\label{sec:initial}

\subsection{A partition of the projective line}

We begin with a partition of the projective line into two points of $\mathcal H(1,q^2)$ and $q-1$ Baer sublines that have exactly one point in common with $\mathcal H(1,q^2)$.

\begin{theorem}\label{thm:line}
The points of $\PG(1,q^2)$ can be partitioned into two points of
$\mathcal H(1,q^2)$ and $q-1$ Baer sublines, each meeting
$\mathcal H(1,q^2)$ in exactly one point.

\end{theorem}

\begin{proof}
In the elliptic-quadric model of the inversive plane
(see \cite{Hirschfeld}), the points of $\PG(1,q^2)$ correspond
to the points of $\cQ^-(3,q)$, and Baer sublines correspond to
its plane conic sections. In particular, the $q+1$ points of
$\mathcal H(1,q^2)$ correspond to the $q+1$ points of a conic
$C$ of the elliptic quadric $\cQ^-(3,q)$.
Let $P_1,P_2$ be two fixed points of $C$, and let $\pi_i$ be the
plane tangent to $\cQ^-(3,q)$ at $P_i$, for $i=1,2$.
Then the line $r=\pi_1\cap\pi_2$ is external to $\cQ^-(3,q)$
and passes through the nucleus $N$ of $C$.
The remaining $q-1$ planes through $r$ meet
$\cQ^-(3,q)$ in conics $C_1,\ldots,C_{q-1}$.
The plane of $C$ does not contain $r$. Each plane of a conic $C_i$ therefore meets the plane of $C$ in a line through $N$, which meets $C$ in exactly one point. Thus $|C_i\cap C|=1$ for $i=1,\ldots,q-1$.
The plane sections in this pencil partition the elliptic quadric
into $P_1,P_2$ and $C_1,\ldots,C_{q-1}$, giving the required
partition. The two points may be prescribed arbitrarily
on $\mathcal H(1,q^2)$.
\end{proof}

\subsection{A cyclic partition of the projective plane}

We next give a cyclic construction of an analogous partition of the projective plane.

Let $q$ be an even prime power. We use the cyclic model
\[
\mathcal P=\{P_x=(x,x^{q^2},x^{q^4}):x\in \F_{q^6}^*\}\cong\PG(2,q^2),
\]
where $P_x=P_y$ if and only if $y/x\in \F_{q^2}^*$.
The auxiliary plane is
\[
\pi=\{p_y=(y,y^{q^2},y^q):y\in \F_{q^3}^*\}\cong\PG(2,q),
\]
with $p_y=p_z$ if and only if $z/y\in \F_q^*$. The norm $N(x)=x^{q^3+1}$ induces the map
\[
\begin{aligned}
\nu:\mathcal P&\longrightarrow\pi,\\
P_x&\longmapsto p_{N(x)}
=\bigl(x^{q^3+1},x^{q^5+q^2},x^{q^4+q}\bigr).
\end{aligned}
\]
We take
\[
\mathcal H(2,q^2)=\{P_x:\Tr_{q^3|q}(N(x))=0\}.
\]
This is a nondegenerate Hermitian curve of $\mathcal P$: its defining Hermitian form is given by $\Tr_{q^6|q^2}(xy^{q^3})$, where $\Tr$ denotes the trace function. For background on cyclic models of Hermitian curves, see~\cite{BE}. Thus $\mathcal H(2,q^2)=\nu^{-1}(h)$, where
\[
h=\{p_y:\Tr_{q^3|q}(y)=0\}
\]
is a line of $\pi$.

For $c\in\F_{q^6}^*$ with $N(c)=1$, multiplication by $c$ induces the projectivity
\[
g_c:P_x\longmapsto P_{cx},\qquad
(x,x^{q^2},x^{q^4})\longmapsto
(cx,c^{q^2}x^{q^2},c^{q^4}x^{q^4}).
\]
Since multiplication by $c$ is an invertible $\F_{q^2}$-linear map on $\F_{q^6}$, the map $g_c$ fixes $\mathcal P$ setwise. The norm-one elements form a cyclic group of order $q^3+1$, whose kernel on $\mathcal P$ consists of the $q+1$ elements of $\F_{q^2}^*$ having norm one. Thus the induced group is
\[
G=\{c\in\F_{q^6}^*:N(c)=1\}/\{c\in\F_{q^2}^*:N(c)=1\},
\qquad |G|=q^2-q+1.
\]
Moreover, $N(cx)=N(c)N(x)=N(x)$, and hence
\[
\Tr_{q^3|q}(N(cx))=\Tr_{q^3|q}(N(x)).
\]
Consequently, $G$ also fixes $\mathcal H(2,q^2)$ setwise and satisfies $\nu\circ g=\nu$ for every $g\in G$.

The group $G$ acts semiregularly on $\mathcal P$. Indeed, if $g_c(P_x)=P_x$, then $c\in\F_{q^2}^*$, so $g_c$ is the identity.

For every $p_y\in\pi$, the set $\nu^{-1}(p_y)$ is a $G$-orbit of size $q^2-q+1$. Indeed, if $\nu(P_x)=\nu(P_z)$, then $N(z)/N(x)\in\F_q^*$. As squaring permutes $\F_q^*$, one can choose $\lambda\in\F_q^*$ such that $N(\lambda z)=N(x)$. The element $c=\lambda z/x$ then has norm one and satisfies $g_c(P_x)=P_z$.

\begin{theorem}\label{thm:plane}
The plane $\mathcal P$ admits a partition into $q(q^2-q+1)$ Baer sublines having one point in common with $\mathcal H(2,q^2)$ and $q^2-q+1$ points of $\mathcal H(2,q^2)$, invariant under $G$.
\end{theorem}

\begin{proof}
Choose $a,b\in \F_{q^3}$ such that
\[
a\ne0,\qquad \Tr_{q^3|q}(a)=0,\qquad \Tr_{q^3|q}(b)=\Tr_{q^3|q}(ab)=1,
\]
and choose $\delta\in \F_q$ satisfying
\[
\Tr_{q|2}(\delta)=1.
\]
Such $a,b$ exist: a nonzero trace-zero element $a$ does not belong to $\F_q$, and the $\F_q$-linear map
\[
b\in\F_{q^3}\longmapsto
\bigl(\Tr_{q^3|q}(b),\Tr_{q^3|q}(ab)\bigr)\in\F_q^2
\]
is surjective, since the linear forms $b\mapsto\Tr_{q^3|q}(b)$ and $b\mapsto\Tr_{q^3|q}(ab)$ are linearly independent over $\F_q$.

The elements $a^2,ab,a$ form an $\F_q$-basis of $\F_{q^3}$. Indeed, a relation $\alpha a+\beta b+\gamma=0$, with $\alpha,\beta,\gamma\in\F_q$, first gives $\beta+\gamma=0$ upon taking traces, and then $\beta=0$ upon multiplying by $a$ and taking traces; hence all coefficients vanish. In the corresponding coordinates
\[
p_{a^2X+abY+aZ}\longleftrightarrow(X,Y,Z)
\]
of $\pi$, we have $h:Y=0$. Set
\[
p_a=(0,0,1),\qquad \ell:Z=0,
\]
and, for $s\in\F_q^*$, let $C_s$ be the non-degenerate conic of $\pi$ given by
\[
X^2+XY+\delta^2Y^2=s^2Z^2.
\]
These $q-1$ conics are pairwise disjoint, have common nucleus $p_a$, and each meets $h$ in exactly one point. They are also disjoint from $\ell$. Hence
\[
\{p_a\}\ \cup\ \ell\ \cup\ \left(\bigcup_{s\in\F_q^*}C_s\right)
\]
is a partition of $\pi$.

Let $B_\infty$ be the Baer subline of $\mathcal P$ given by
\[
B_\infty=\{P_{ra+t\sqrt{ab}}:(r,t)\in\PG(1,q)\}.
\]
The defining vectors are independent over $\F_{q^2}$, since $\F_{q^3}\cap \F_{q^2}=\F_q$ and $a^2,ab$ are independent over $\F_q$. As they belong to $\F_{q^3}$,
\[
N(ra+t\sqrt{ab})=a^2r^2+ab\,t^2,\qquad r,t\in\F_q,
\]
so $\nu$ maps $B_\infty$ bijectively onto $\ell$.

For $s\in\F_q^*$, the polynomial
\[
f(Z)=a(1+sa)Z^2+aZ+a(\delta+sb)
\]
is irreducible over $\F_{q^3}$. Indeed,
\[
\begin{aligned}
\Tr_{q^3|2}\bigl((1+sa)(\delta+sb)\bigr)
&=\Tr_{q|2}\!\left(\Tr_{q^3|q}\bigl((1+sa)(\delta+sb)\bigr)\right)\\
&=\Tr_{q|2}(\delta+s+s^2)\\
&=1.
\end{aligned}
\]
Choose $z_s\in \F_{q^6}\setminus \F_{q^3}$ such that $f(z_s)=0$. Then $z_s+z_s^{q^3}=1/(1+sa)$ and $z_s^{q^3+1}=(\delta+sb)/(1+sa)$. Set $u_s\in \F_{q^6}^*$ with $N(u_s)=a(1+sa)$, and $v_s=u_sz_s$. Then
\[
\begin{aligned}
N(ru_s+tv_s)
&=a(1+sa)r^2+art+a(\delta+sb)t^2\\
&=a^2sr^2+ab\,st^2+a(r^2+rt+\delta t^2),\qquad r,t\in\F_q.
\end{aligned}
\]
The vectors $u_s,v_s$ are independent over $\F_{q^2}$, otherwise $z_s\in \F_{q^2}$, implying $z_s+z_s^{q^3}=z_s+z_s^q=1/(1+sa)\in\F_q$, which is impossible since $a\notin\F_q$. Hence
\[
B_s=\{P_{ru_s+tv_s}:(r,t)\in\PG(1,q)\}
\]
is a Baer subline. In the chosen coordinates of $\pi$, we have
\[
\nu(B_s)=\{(sr^2,st^2,r^2+rt+\delta t^2):(r,t)\in\PG(1,q)\}=C_s,
\]
and this correspondence is bijective.

Whenever $\nu$ maps a subline $B$ bijectively onto a set $D$, the action of $G$ on each set $\nu^{-1}(p_y)$, with $p_y\in D$, yields
\[
\nu^{-1}(D)=\bigcup_{g\in G}g(B).
\]
Indeed, $B$ meets $\nu^{-1}(p_y)$ in exactly one point for each $p_y\in D$, and its $G$-translates partition $\nu^{-1}(D)$. To see disjointness explicitly, if $g(P)=g'(P')$ with $P,P'\in B$, applying $\nu$ gives $P=P'$ by injectivity on $B$; semiregularity then gives $g=g'$.

Let
\[
S=\nu^{-1}(p_a)=\{P_x:N(x)/a\in\F_q^\star\}.
\]
Since $\Tr_{q^3|q}(a)=0$, $S\subseteq \mathcal H(2,q^2)$, and $|S|=q^2-q+1$. The required partition is
\begin{equation}\label{eq:partition}
S\ \cup\ \{g(B_\infty):g\in G\}
\ \cup\ \{g(B_s):s\in \F_q^*,\ g\in G\}.
\end{equation}
Disjointness and completeness follow from the partition of $\pi$ given by $\{p_a\}\cup\ell\cup\bigl(\bigcup_{s\in\F_q^*}C_s\bigr)$ and the preceding observation.

Finally,
\[
\Tr_{q^3|q}(N(ra+t\sqrt{ab}))=t^2,\qquad
\Tr_{q^3|q}(N(ru_s+tv_s))=s\,t^2.
\]
Thus every initial subline meets $\mathcal H(2,q^2)$ exactly at the point corresponding to $(r,t)=(1,0)$. Since $G$ preserves $\mathcal H(2,q^2)$, every subline in \eqref{eq:partition} is unisecant to $\mathcal H(2,q^2)$. There are $q(q^2-q+1)$ such sublines and $q^2-q+1$ points, as asserted.
\end{proof}

\subsection{Higher dimensions}\label{sec:induction}

Here we extend the preceding constructions to higher dimensions.

In $\PG(2,q^2)$, let $\ell$ be a line and let $P$ be a point of
$\PG(2,q^2)$, with $P\notin\ell$. Let $\mathcal H(1,q^2)$ be a Hermitian variety of
$\ell$, and let $\ell_0$ be a Baer subline of $\ell$ such that
$\ell_0\cap\mathcal H(1,q^2)=\{Q\}$. Denote by $\mathcal C_S$
the cone with vertex $P$ and base the point set $S\subseteq\ell$.

\begin{lemma}\label{lem:cone-lifting}
The cone $\mathcal C_{\ell_0}\setminus\{P\}$ can be partitioned
into $q^2$ Baer sublines. Every subline in the partition meets
$\mathcal C_{\mathcal H(1,q^2)}$ in exactly one point.
\end{lemma}
\begin{proof}
Choose $R\in\ell\setminus\ell_0$, and let $\mathcal T_R$ be
the pencil of lines through $R$ with $PR$ omitted. It has $q^2$
members, none containing $P$. For $t\in\mathcal T_R$, put
$\mathcal B_t=\{PX\cap t:X\in\ell_0\}$.
The set $\{\mathcal B_t:t\in\mathcal T_R\}$ consists of $q^2$
pairwise disjoint Baer sublines partitioning
$\mathcal C_{\ell_0}\setminus\{P\}$.
Finally, $\ell_0\cap\mathcal H(1,q^2)=\{Q\}$ implies
$\mathcal C_{\ell_0}\cap\mathcal C_{\mathcal H(1,q^2)}=PQ$.
Each $\mathcal B_t$ meets the generator $PQ$ exactly once, and hence
$\mathcal B_t\cap\mathcal C_{\mathcal H(1,q^2)}=\{PQ\cap t\}$.
\end{proof}

\begin{theorem}\label{thm:extension}
Let $q$ be even and $n\ge4$. Let $\mathcal H(n-1,q^2)$ be a
non-degenerate Hermitian variety of $\Pi=\PG(n-1,q^2)$ and let
$\mathcal H(n-3,q^2)$ be a non-degenerate Hermitian variety of
$\PG(n-3,q^2)$. If $\PG(n-3,q^2)$ has a partition into
\[
 \frac{(q^{n-4}+(-1)^n)(q^{n-2}-(-1)^n)}{q^2-1}
\]
points of $\mathcal H(n-3,q^2)$ and
\[
 \frac{q^{n-4}(q^{n-2}-(-1)^n)}{q+1}
\]
Baer sublines, each meeting $\mathcal H(n-3,q^2)$ in exactly one
point, then $\Pi$ has a partition into
\[
 \frac{(q^{n-2}+(-1)^n)(q^n-(-1)^n)}{q^2-1}
\]
points of $\mathcal H(n-1,q^2)$ and
\[
 \frac{q^{n-2}(q^n-(-1)^n)}{q+1}
\]
Baer sublines, each meeting $\mathcal H(n-1,q^2)$ in exactly one point.
\end{theorem}
\begin{proof}
Choose $P\in\mathcal H(n-1,q^2)$ and let $T$ be the tangent
hyperplane to $\mathcal H(n-1,q^2)$ at $P$. Choose a subspace
$S\cong\PG(n-3,q^2)$ such that $T=\langle P,S\rangle$ and
$\mathcal H(n-1,q^2)\cap T$ is the cone with vertex $P$ and
non-degenerate base $\mathcal H(n-1,q^2)\cap S$, which we identify
with $\mathcal H(n-3,q^2)$. By hypothesis $S$ has a partition into
a subset $\mathcal A_0$ of
$(q^{n-4}+(-1)^n)(q^{n-2}-(-1)^n)/(q^2-1)$ points of
$\mathcal H(n-3,q^2)$ and a set $\mathcal B_0$ of
$q^{n-4}(q^{n-2}-(-1)^n)/(q+1)$ Baer sublines, each meeting
$\mathcal H(n-3,q^2)$ in exactly one point. We want to show the
existence of a partition of $\Pi$ consisting of a subset
$\mathcal A$ of $(q^{n-2}+(-1)^n)(q^n-(-1)^n)/(q^2-1)$ points
of $\mathcal H(n-1,q^2)$ and a set $\mathcal B$ of
$q^{n-2}(q^n-(-1)^n)/(q+1)$ Baer sublines, each meeting
$\mathcal H(n-1,q^2)$ in exactly one point.

For each $Q\in\mathcal A_0$, include all the $q^2+1$ points of the line
$PQ$ in $\mathcal A$. For each Baer subline
$\ell_0\in\mathcal B_0$, by Lemma~\ref{lem:cone-lifting}, the cone
$\mathcal C_{\ell_0}\setminus\{P\}$ can be partitioned into $q^2$
Baer sublines, each of them meeting $\mathcal H(n-1,q^2)$ in
exactly one point. For each $\ell_0\in\mathcal B_0$, include these
$q^2$ Baer sublines in $\mathcal B$. Thus $T$ is partitioned into
$1+q^2|\mathcal A_0|$ points of $\mathcal H(n-1,q^2)$ and $q^2|\mathcal B_0|$
Baer sublines, each meeting $\mathcal H(n-1,q^2)$ in exactly one
point.

For each line $\ell$ through $P$ not contained in $T$, the line is
secant to $\mathcal H(n-1,q^2)$, and
$\ell\cap\mathcal H(n-1,q^2)$ is a Baer subline containing
$P$. Choose $R_\ell\in (\ell\cap\mathcal H(n-1,q^2))\setminus\{P\}$. By Theorem~\ref{thm:line}, there is a partition of $\ell$ into the
two points $P,R_\ell$ of $\mathcal H(n-1,q^2)$ and $q-1$ Baer
sublines, each
meeting $\mathcal H(n-1,q^2)$ in exactly one point. Include $R_\ell$ in $\mathcal A$ and these $q-1$
Baer sublines in $\mathcal B$. Distinct lines
through $P$ not contained in $T$ meet only at $P$, and every point
of $\Pi\setminus T$ belongs to exactly one such line. Thus these
partitions are mutually disjoint and cover the complement of $T$.

There are $q^{2n-2}$ points outside $T$, and each line through $P$
not contained in $T$ has $q^2$ such points. Hence there are
$q^{2n-4}$ such lines. Each contributes with one point of $\mathcal A$ and $q-1$
Baer sublines. Adding these contributions to the partition of $T$
gives
\[
 |\mathcal A|=1+q^2|\mathcal A_0|+q^{2n-4},\qquad
 |\mathcal B|=q^2|\mathcal B_0|+(q-1)q^{2n-4}.
\]
Substituting the values of $|\mathcal A_0|$ and $|\mathcal B_0|$ yields the stated
formulas.
\end{proof}

The preceding extension theorem now yields the partition in every
projective dimension.

\begin{theorem}\label{thm:hermitian-partition}
Let $q$ be even and $n\ge 2$, and let $\mathcal H(n-1,q^2)$ be a
non-degenerate Hermitian variety in $\PG(n-1,q^2)$. Put
$\varepsilon=(-1)^n$. Then there exists a partition of
$\PG(n-1,q^2)$ into
\[
\frac{(q^{n-2}+\varepsilon)(q^n-\varepsilon)}{q^2-1}
\]
points of $\mathcal H(n-1,q^2)$ and
\[
\frac{q^{n-2}(q^n-\varepsilon)}{q+1}
\]
Baer sublines, each meeting $\mathcal H(n-1,q^2)$ in exactly one point.
\end{theorem}
\begin{proof}
By induction on $n$, separately for even and odd $n$, the result follows. Theorem~\ref{thm:line}
provides the case $n=2$, with two points and $q-1$ Baer sublines;
Theorem~\ref{thm:plane} provides the case $n=3$, with
$q^2-q+1$ points and $q(q^2-q+1)$ Baer sublines. For $n\ge4$,
apply Theorem~\ref{thm:extension} to the partition of $\PG(n-3, q^2)$. This completes the induction.
\end{proof}

Combining Theorem~\ref{thm:hermitian-partition} and Theorem~\ref{thm:replacement}, a line spread of $\mathcal W(2n-1,q)$ for every even prime power $q$ is obtained.

\begin{corollary}\label{cor:existence}
For every even prime power $q$ and every $n\ge2$, the symplectic polar
space $\cW(2n-1,q)$ admits a line spread obtained by simultaneous
opposite-regulus replacement in a Desarguesian line spread.
\end{corollary}

\section*{Acknowledgments}
The authors acknowledge support from the Italian National Group for Algebraic and Geometric Structures and their Applications (GNSAGA--INdAM).


\begin{thebibliography}{99}
\bibitem{DeWinterMetsch2020}
S.~De Winter and K.~Metsch,
\emph{Perfect 2-colorings of the Grassmann graph of planes},
Electron. J. Combin. \textbf{27} (2020), Paper P1.21.
\doi{10.37236/8672}.

\bibitem{Hering1985}
C.~Hering,
\emph{Two new sporadic doubly transitive linear spaces},
in Finite Geometries (Winnipeg, 1984),
Lecture Notes in Pure and Applied Mathematics, vol.~103,
Dekker, New York, 1985, pp.~127--129.

\bibitem{Thas1981}
J.~A.~Thas,
\emph{Ovoids and spreads of finite classical polar spaces},
Geom. Dedicata \textbf{10} (1981), 135--143.
\doi{10.1007/BF01447417}.

\bibitem{BE}
S.~Barwick and G.~Ebert,
\emph{Unitals in Projective Planes},
Springer Monographs in Mathematics, Springer, New York, 2008.
\doi{10.1007/978-0-387-76366-8}.

\bibitem{Hirschfeld}
J.~W.~P.~Hirschfeld,
\emph{Finite Projective Spaces of Three Dimensions},
Oxford Mathematical Monographs, Clarendon Press, Oxford, 1985.

\bibitem{HT}
J.~W.~P.~Hirschfeld and J.~A.~Thas,
\emph{General Galois Geometries}, second edition,
Springer Monographs in Mathematics, Springer, London, 2016.
\doi{10.1007/978-1-4471-6790-7}.

\bibitem{LV}
M.~Lavrauw and G.~Van de Voorde,
\emph{Field reduction and linear sets in finite geometry},
in Topics in Finite Fields, Contemporary Mathematics, vol.~632,
American Mathematical Society, 2015.

\bibitem{PT}
S.~E.~Payne and J.~A.~Thas,
\emph{Finite Generalized Quadrangles}, second edition,
EMS Series of Lectures in Mathematics,
European Mathematical Society, Z\"urich, 2009.
\doi{10.4171/066}.
\end{thebibliography}
\end{document}